\documentclass[11pt,oneside,english,reqno]{amsart}

\usepackage{xcolor}
\definecolor{stronggreen}{HTML}{009900}
\usepackage{amsmath,amssymb,amsthm}
\usepackage{geometry}
\usepackage{hyperref}
\usepackage{enumitem}
\usepackage{blindtext}
\usepackage{graphicx}
\usepackage[utf8]{inputenc}
\usepackage[normalem]{ulem}
\newtheorem{theorem}{Theorem}[section]

\newtheorem{remark}{Remark}[section]
\newtheorem{lemma}{Lemma}[section]
\newtheorem{proposition}{Proposition}[section]
\newtheorem{corollary}{Corollary}[section]
\newcommand{\R}{\mathbb{R}}

\newenvironment{acknowledgements}%
  {\section*{Acknowledgments}}%
  {}

\usepackage{xcolor}

\title[Free boundary for  quasilinear systems]{\bf A Free Boundary Problem for  quasilinear Systems  \\ with mixed variable exponent}

\author{Somayeh Khademloo}


\begin{document}
\maketitle

\begin{abstract}
In the unit ball, we study a semilinear system that gives rise to a free boundary of Alt-Phillips type, where the equation is governed by a mixed variable-exponent operator of $(p,q)$-Laplacian type. 

Under certain structural conditions, we establish the existence of nonnegative radial solutions that are $C^1$, with their norms depending on the structural data.

\end{abstract}

\tableofcontents
\section{Introduction and problem formulation}
\subsection{Background}

In this article, we consider a system of strongly coupled quasilinear system in the unit ball with constant boundary values, which covers both regular and singular right-hand sides; see equation \eqref{eq:main-system-variable}. Our problem, due to the right-hand-side absorption term, gives rise to a free boundary of the so-called Alt-Phillips type, provided the Dirichlet data is not large.
Such problems in scalar case are well studied for a wide range of operators, including  fully nonlinear and $p$-Laplace operators; see e.g.  
\cite{AlamriUrbano2026, AltPhi86, AraTei13} and the references therein.

Continuing along this line, we study a two-component case with two different variable exponents, as given in equation \eqref{eq:main-system-variable}.  
 This model problem, in the case of constant exponents and systems of $m$ components, has been studied in \cite{BiagiValdinociVecchi2020,ElShah,MDS-Nor}.
 While the authors in \cite{ElShah} assume the existence of certain barriers to prove the existence of solutions, in the present work, we are able to circumvent this assumption by applying a direct method to prove the existence of solutions.
 
In recent years (see \cite{Ferra, Juli, Led, Harj2007}), there has been increasing interest in free boundary problems with \emph{variable exponent growth}, where the diffusion is governed by operators of the form
\[
\mathrm{div}\big(|\nabla u|^{p(x)-2}\nabla u\big).
\]
Indeed,  variable exponent models are important for capturing physical processes where material properties change according to local conditions. Common examples include electrorheological fluids \cite{Diening2011, Ruzicka2000}, where viscosity shifts in response to electric fields, and image restoration \cite{Chen2006}, where the exponent $p(x)$ is tuned to preserve edges while smoothing noise. These problems share a common structure: the free boundary marks the transition between distinct physical states. While the $p(x)$-Laplacian is well-suited to model this adaptation, it also discards the homogeneity and scaling properties we rely on in standard elliptic theory. This loss of structure is precisely what makes the analytical treatment of these free boundaries so challenging.

It should be mentioned that El Hajj-Jeon-Shahgholian \cite{El-Je-Sh-2} have recently extended the result of \cite{ElShah} to more general domains, for $p$-Laplacian as well as fully nonlinear operators. Their approach has a good chance of working for our problem, but it would require developing the scalar case of the $(p,q)$-Laplacian of Alt-Phillips problem as a first stage, which is currently not available in the literature.

\subsection{Problem setting}
Let $B_1$ denote the unit ball in $\R^n$ ($n \geq 2$,) and 
\[
p,q:\overline{B_1}\times [0,\infty)\times [0,\infty)\to (1,\infty)
\]
be radial  functions--representing exponents--in the sense that for any radial function $w$,
the quantities
\[
p(x,w(x),|\nabla w(x)|),\qquad q(x,w(x),|\nabla w(x)|) ,
\]
depend only on $|x|$, $w(|x|)$, and $|w'(|x|)|$.

We denote the quasilinear operator with mixed variable exponents by\footnote{Observe that the coupling of the system is only on the right-hand side and not on the PDE.  }
\[
\mathcal L_{p,q} w:=\mathcal L_p w+\mathcal L_q w,
\]
where
\[
\mathcal L_s w
:=
\nabla\!\cdot\!\Big(|\nabla w|^{\,s(x,w(x),|\nabla w(x)|)-2}\nabla w\Big),\qquad s\in\{p,q\},
\]
and consider a two-component   system of free boundaries
\begin{equation}\label{eq:main-system-variable}
\begin{cases}
\mathcal L_{p,q} u
=
h(|x|,u,v)\,\mathbf 1_{\{u>0\}}
& \text{in } B_1,\\[4pt]
\mathcal L_{p,q} v
=
k(|x|,u,v)\,\mathbf 1_{\{v>0\}}
& \text{in } B_1,\\[4pt]
u=M_1,\qquad v=M_2
& \text{on } \partial B_1,
\end{cases}
\end{equation}
where $M_1,M_2>0$ are prescribed constants and $\mathbf  1_D$ is the characteristic function of a domain $D$. The functions $h,k$ on the right-hand side of \eqref{eq:main-system-variable} will satisfy a certain structure, which we explain below. 

The corresponding free boundaries (which we expect to coincide in our radial case) are denoted by
\[
\partial\{u>0\},
\qquad
\partial\{v>0\}.
\]

Since all the ingredients in this paper are assumed to be radial, we expect any solution to be radial, although we cannot prove that this is the case of variably exponent, while this is true for the constant-coefficient case, as shown in \cite{ElShah}.  Note that the classical method of moving planes cannot be applied directly in the presence of variable coefficients, even when the exponents depend solely on the spatial variable $|x|$. In that setting, reflecting the domain across a moving hyperplane alters the localized values of the exponents, preventing the reflected and original operators from matching. In our framework, the situation is more delicate as the exponents $p(x, w, |\nabla w|)$ and $q(x, w, |\nabla w|)$ depend simultaneously on the spatial coordinates, the solution states, and the gradient magnitudes. This structural coupling breaks both standard scaling invariance and reflection symmetry, making the question of radial symmetry for general variable-coefficient problems a challenging open problem. 

\medskip
\noindent

In a  radial setting, our problem can be formulated as
\begin{equation}\label{eq:main-system-radial}
\begin{cases}
\dfrac{1}{r^{n-1}}\dfrac{d}{dr}\!\left(
r^{n-1}\Big(|u'|^{p(r,u,|u'|)-2}u'+|u'|^{q(r,u,|u'|)-2}u'\Big)
\right)
=
h(r,u,v)\,\mathbf 1_{\{u>0\}},
\\[10pt]
\dfrac{1}{r^{n-1}}\dfrac{d}{dr}\!\left(
r^{n-1}\Big(|v'|^{p(r,v,|v'|)-2}v'+|v'|^{q(r,v,|v'|)-2}v'\Big)
\right)
=
k(r,u,v)\,\mathbf 1_{\{v>0\}},
\end{cases}
\end{equation}
for $r\in(0,1)$.

\medskip
\noindent

We shall further impose the following standing assumptions on the variable exponents and on the sink, i.e.,  terms on the right-hand side.

\medskip

\noindent
\underline{\bf Standing Structural Assumptions:}

\medskip

\noindent
\textbf{(P1) Exponents.}

The mappings 
\[
p, q : [0,1] \times [0,\infty) \times [0,\infty) \to (1,\infty)
\]
are assumed to be locally $C^\alpha $  ($\alpha \in (0,1]$)  with respect to the first two variables on compact subsets, and continuous with respect to the third variable. Moreover, there exist global structural constants such that
\[
1 < p_- \le p(r,s,\xi) \le p_+ < \infty, \qquad 1 < q_- \le q(r,s,\xi) \le q_+ < \infty
\]
for all $(r,s,\xi) \in [0,1] \times [0,\infty) \times [0,\infty)$.
We also define the uniform exponent bounds  by
\[
1 < p_* := \min\{p_-, q_-\} \le \max\{p_+, q_+\} := p^* < \infty.
\]

\medskip
\noindent
\textbf{(P2) Radial monotonicity and structural power regimes near the origin.}

For the functions
\[
h,k:[0,1]\times[0,\infty)^2\to[0,\infty)
\]there exist structural constants $0<\lambda_0\le \Lambda_0<\infty$ and an admissible power pair $(\gamma_1, \gamma_2) \in \mathbb{R}^2$ satisfying exactly one of the following structural parameter regimes
\begin{enumerate}
    \item[\textup{(i)}] \textit{The regular regime:} 
    $0 \leq  \gamma_1, \gamma_2 < 1$;
    
    \item[\textup{(ii)}] \textit{The semi-singular regime:}
    $ -1 < \min\{\gamma_1, \gamma_2\} <0<  \max\{\gamma_1, \gamma_2\} <  1 $;
    \item[\textup{(iii)}] \textit{The fully singular regime:}  $\gamma_1 + \gamma_2 > -1$, and $ \gamma_1, \gamma_2 < 0$.
\end{enumerate}

When $\gamma_1 = \gamma_2 = 0$, the system is decoupled, and we have two scalar problems, which are not of interest in this paper.

 We impose the following condition on the sink-terms $h,k$. 
 For all sufficiently small state values 
 $\varepsilon_1, \varepsilon_2 > 0$ and all $r \in (0,1]$, the  right-hand sides  satisfy the two-sided  growth restrictions
\begin{equation}\label{eq:hk-structure}
\lambda_0\,\varepsilon_1^{\gamma_1}\varepsilon_2^{\gamma_2} \le h(r,\varepsilon_1,\varepsilon_2) \le \Lambda_0\,\varepsilon_1^{\gamma_1}\varepsilon_2^{\gamma_2},
\qquad 
\lambda_0\,\varepsilon_1^{\gamma_1}\varepsilon_2^{\gamma_2} \le k(r,\varepsilon_1,\varepsilon_2) \le \Lambda_0\,\varepsilon_1^{\gamma_1}\varepsilon_2^{\gamma_2}.
\end{equation}
With these restrictions, it also follows that 
\begin{equation}\label{eq:hk-structure1}
\frac{\lambda_0}{\Lambda_0} h(r,\varepsilon_1,\varepsilon_2) \le 
k(r,\varepsilon_1,\varepsilon_2) \le \frac{\Lambda_0}{\lambda_0} h(r,\varepsilon_1,\varepsilon_2).
\end{equation}

\begin{remark}

The free boundary problem \eqref{eq:main-system-variable} presents several distinct mathematical challenges that differentiate it from classical free boundary problems. Here we shall point out these aspects to some extent.

\begin{itemize}

\item[(i)] \textbf{Extending the range of $\gamma_1, \gamma_2$:}  We allow a larger class of right-hand sides as specified in \textup{(P2)} above. The range of powers includes regular, semi-singular, and fully singular regimes.

\item[(ii)] \textbf{Lack of $C^2$ Smoothness:} When at least one of the powers $\gamma_j$ is negative, the solutions $u$ and $v$ fail to be $C^2$ up to the free boundary. This introduces substantial difficulty in the analysis, especially since the problem is non-variational.

\item[(iii)] \textbf{Non-Homogeneity of the Operator:} Unlike the standard $p$-Laplacian, our operator $\mathcal{L}_{p,q}$ involves mixed variable exponents $p, q(x, w, |\nabla w|)$ that depend on the state and the gradient. This lack of homogeneity implies that the operator does not scale uniformly. For this reason, the usual scaling arguments and explicit barrier constructions cannot be applied naturally.
Our approach therefore avoids barrier-based constructions presented in \cite{ElShah} and instead relies on compactness and structural estimates intrinsic to the problem.

\end{itemize}

\end{remark}

\subsection{Main result}

The main result of this work is the following.

\begin{theorem}\label{thm:interior-pq-variable}
Assume that assumptions \textup{(P1)--(P2)} hold.
Then system  \eqref{eq:main-system-variable} admits at least a radial solution $(u,v)$ with the following properties:

The two components are radially symmetric and radially nondecreasing, i.e.,
\[
u(x)=u(|x|),\qquad v(x)=v(|x|),
\]
and
\[
\partial_r u(r)\ge 0,\qquad \partial_r v(r)\ge 0
\qquad \text{for } r\in(0,1).
\]
Moreover, if the boundary data $M_1,M_2>0$ are sufficiently small, then the (common) coincidence set is nonempty, namely
\[
\{u=0\}=\{v=0\}\neq\varnothing.
\]
\end{theorem}

\begin{remark}
That the supports of both components should coincide sounds reasonable in the present regime.
Indeed, if one component were to vanish while the other remained positive, then the singular structure in \textup{(P2)} would force the corresponding right-hand side to become incompatible with a separated positivity region.
This mechanism is specific to singular couplings of negative power type and may fail for smooth or weakly coupled source terms. 

This may, potentially, fail in general 
(non-symmetric) settings, see \cite{El-Je-Sh-2}.

\end{remark}

\section{Proof of  Theorem \ref{thm:interior-pq-variable}}

\subsection{Iterative scheme for an approximated problem}

Since the right-hand side may exhibit singularities or degeneracies, we approximate it with functions that are neither singular nor degenerate.
We thus  fix a parameter
\[
0<\varepsilon\ll 1,
\]
and define the approximate  nonlinearities
\begin{equation}\label{eq:epsilon-approx}
    h_\varepsilon(r,u,v):=h\big(r,\max\{u,\varepsilon\},\max\{v,\varepsilon\}\big),
\qquad 
k_\varepsilon(r,u,v):=k\big(r,\max\{u,\varepsilon\},\max\{v,\varepsilon\}\big).
\end{equation}
By construction, for $u \approx 0$ and/or  $v\approx 0$  we have
\[
c_0 \le h_\varepsilon,\; k_\varepsilon \le c_1,
\]
and therefore the approximated problem belongs to the class of standard obstacle problems, albeit with certain non-standard features.

We shall now apply an iterative scheme in which, at each step, we solve a scalar problem whose coefficients and data are inherited from the previous iteration.
We initialize the iteration by 
\[
u_0^\varepsilon\equiv M_1,
\qquad
v_0^\varepsilon\equiv M_2,
\]
and construct inductively a sequence $\{(u_i^\varepsilon,v_i^\varepsilon)\}_{i\ge1}$.

For $i\ge1$, assume that $(u_{i-1}^\varepsilon,v_{i-1}^\varepsilon)$ is already known, and define the  exponents
\[
s_i^w(r):=s\bigl(r,w_{i-1}^\varepsilon(r),|(w_{i-1}^\varepsilon)'(r)|\bigr),
\qquad
\textrm{for} ~s\in\{p,q\},\quad w\in\{u,v\}.
\]
By \textup{(P1)}, for each $i\ge1$, $s_i^w(r)$ is a measurable function satisfying $1 < s_i^w(r) < \infty.
$

Next, we  define 
\[
0 \leq f_i^\varepsilon(r):=
h_\varepsilon\bigl(r,u_{i-1}^\varepsilon(r),v_{i-1}^\varepsilon(r)\bigr) \in L^\infty((0,1)),
\]
and consider the scalar obstacle-type problem, as a minimizer of the functional  
\begin{equation}
   \mathcal J_i^\varepsilon(w)
:=
\int_{0}^1
\left(
\frac{1}{p_i^u(r)}| w'(r)|^{p_i^u(r)}
+
\frac{1}{q_i^u(r)}| w'(r)|^{q_i^u(r)}+f_i^\varepsilon(r)\,w(r)
\right)\,r^{n-1}dr,
\end{equation} 
on the associated admissible class
\begin{equation}
\mathcal K_i^{u,\varepsilon}
:=
\Bigl\{
w\in \textbf{W}^{u,i},\
w\ge0 \text{ a.e. in }(0,1),\ \ w(1)=M_1
\Bigr\},
\end{equation}
where \[\textbf{W}^{u,i}:=W^{1,p_i^u(\cdot)}\bigl((0,1);r^{n-1}dr\bigr)
\cap
W^{1,q_i^u(\cdot)}\bigl((0,1);r^{n-1}dr\bigr).\]

\medskip

\noindent

For notational simplicity, we suppress the dependence of the functional on functions from the previous iterations and the corresponding coefficients. Indeed, at each fixed iteration step, these quantities are regarded as prescribed coefficients. The variational problem is therefore understood as a minimization problem only with respect to the variable $w$.

Note that in the radial class, the one-dimensional interval $(0, 1)$  in \eqref{eq:main-system-radial} should be 
understood as the radial reduction of the ball $B_1 \subset \mathbb{R}^n$, 
rather than as an independent one-dimensional boundary value problem. 
In particular, no Dirichlet condition is prescribed at $r=0$. Instead, 
the regularity at the origin implies the natural symmetry condition $w'(0)=0$. 
Nevertheless, since we are mainly interested in proving the existence of a 
free boundary problem—which implies that the set $\{(u,v)=(0,0)\}$ is 
non-empty and a ball centered at the origin (due to expected symmetry)—we 
will not face any technical problems regarding the behavior of our solution 
pair at the origin, or in spherical coordinates at $r=0$.

\medskip

\noindent

To proceed, we  define $u_i^\varepsilon$ as the   unique minimizer of $\mathcal J_i^\varepsilon$ over
$\mathcal K_i^{u,\varepsilon}$, that is,
\[
\mathcal J_i^\varepsilon(u_i^\varepsilon)
=
\min_{w\in \mathcal K_i^{u,\varepsilon}}
\mathcal J_i^\varepsilon(w).
\]

The direct method of the calculus of variations yields the existence of a unique minimizer, since the
admissible set is nonempty, closed, and convex, while the functional is coercive and weakly lower semicontinuous.
Moreover, the strict convexity of the first integral implies uniqueness.

Note that uniqueness is essential here. It ensures that every step in our iteration is well-defined, allowing us to safely use the current solution to build the next one.

Once $u_i^\varepsilon$ is obtained, we define 
\[
g_i^\varepsilon(r):=
k_\varepsilon\bigl(r,u_i^\varepsilon(r),v_{i-1}^\varepsilon(r)\bigr),
\]
and let $v_i^\varepsilon$ be the unique minimizer of
\begin{equation}
    \mathcal I_i^\varepsilon(z)
:=
\int_{0}^1
\left(
\frac{1}{p_i^v(r)}| z'(r)|^{p_i^v(r)}
+
\frac{1}{q_i^v(r)}| z'(r)|^{q_i^v(r)}+g_i^\varepsilon(r)\,z(r)
\right)\,r^{n-1}dr
\end{equation}
over admissible class $\mathcal K_i^{v,\varepsilon}\subset\textbf{W}^{v,i}$ as well.

In this setting, the minimizers satisfy the weighted variational inequalities
\begin{equation}\label{eq:variational-radial-ui-compact}
\int_0^1
\mathcal A_i^u\bigl(r,(u_i^\varepsilon)'(r)\bigr)
(\phi'(r)-(u_i^\varepsilon)'(r))\,r^{n-1}\,dr
+
\int_0^1
f_i^\varepsilon(r)\bigl(\phi(r)-u_i^\varepsilon(r)\bigr)\,r^{n-1}\,dr
\ge 0
\end{equation}
for all $\phi\in \mathcal K_i^{u,\varepsilon}$, and
\begin{equation}\label{eq:variational-radial-vi-compact}
\int_0^1
\mathcal A_i^v\bigl(r,(v_i^\varepsilon)'(r)\bigr)
(\psi'(r)-(v_i^\varepsilon)'(r))\,r^{n-1}\,dr
+
\int_0^1
g_i^\varepsilon(r)\bigl(\psi(r)-v_i^\varepsilon(r)\bigr)\,r^{n-1}\,dr
\ge 0
\end{equation}
for all $\psi\in \mathcal K_i^{v,\varepsilon}$, where the multi-phase operators are defined by
\[
\mathcal A_i^u(r,\xi)
:=
|\xi|^{p_i^u(r)-2}\xi+|\xi|^{q_i^u(r)-2}\xi,
\qquad
\mathcal A_i^v(r,\xi)
:=
|\xi|^{p_i^v(r)-2}\xi+ |\xi|^{q_i^v(r)-2}\xi.
\]
Take into account that the admissible classes are convex, and so \eqref{eq:variational-radial-ui-compact} and \eqref{eq:variational-radial-vi-compact}
follow directly from the first variation of the corresponding convex
functional. Indeed, if $u_i^\varepsilon$ be the minimizer of
$\mathcal J_i^\varepsilon$ over $\mathcal K_i^{u,\varepsilon}$, for
$\phi\in \mathcal K_i^{u,\varepsilon}$ and $0<t<1$, the convexity of
$\mathcal K_i^{u,\varepsilon}$ implies
\(u_i^\varepsilon+t(\phi-u_i^\varepsilon)\in \mathcal K_i^{u,\varepsilon}.
\)
Hence the map
\[
t\mapsto 
\mathcal J_i^\varepsilon\bigl(u_i^\varepsilon+t(\phi-u_i^\varepsilon)\bigr)
\]
has a minimum at $t=0$ from the right. Therefore, its right derivative at
$t=0$ is nonnegative. Since the exponents
$p_i^u,q_i^u$ and the source $f_i^\varepsilon$ are functions of
$r$, differentiation under the integral sign gives \eqref{eq:variational-radial-ui-compact}. The same argument can be used for  $v_i^\varepsilon$.

Under the  assumption $p_* > 1$, $\textbf{W}^{u,i}$ is continuously embedded in $W^{1, p_*}((0,1); r^{n-1}dr)$, which  follows from the one-dimensional Sobolev embedding that $u_i^\varepsilon$ admits a continuous representative on $[0,1]$. The same holds for $v_i^\varepsilon$.

\medskip

\noindent

From the variational inequalities \eqref{eq:variational-radial-ui-compact}--\eqref{eq:variational-radial-vi-compact} (see also \eqref{eq:delta-measure-u}--\eqref{eq:delta-measure-v}), one deduces that $(u_i^\varepsilon)'\ge0$ and $(v_i^\varepsilon)'\ge0$ almost everywhere. Hence both functions are nondecreasing, and their positivity sets are intervals of the form
\begin{equation}\label{eq:interval}
\{r \in (0,1) : u_i^\varepsilon(r) > 0\} = (r_{u_i^\varepsilon}, 1), \qquad
\{r \in (0,1) : v_i^\varepsilon(r) > 0\} = (r_{v_i^\varepsilon}, 1).
\end{equation}

The same variational inequalities also yield the Euler--Lagrange equations satisfied by the minimizers on their positivity sets. However, it is not immediately clear how these equations extend across the free boundary, where the functions vanish, since the distributional derivatives may contain a point mass or even a more singular distribution. Such behavior may occur, for instance, if a minimizer behaves locally like $(r-r_0)_+^a$ with $a\le 1$ near a point $r_0$ at which it becomes zero.

\begin{lemma}\label{lem:delta-measure-radial}
For a fixed $i\ge1$ and $\varepsilon>0$, let $u_i^\varepsilon$ and $v_i^\varepsilon$ be the unique minimizers of $\mathcal J_i^\varepsilon$ and $\mathcal I_i^\varepsilon$ over their respective admissible classes. Then there exist unique contact radii $r_{u_{i}^\varepsilon}, r_{v_{i}^\varepsilon} \in [0,1)$ such that
\begin{equation}\label{eq:delta-measure-u}
\frac{d}{dr}\!\left(
r^{n-1}\mathcal A_i^u\bigl(r,(u_i^\varepsilon)'(r)\bigr)
\right)
=
r^{n-1}f_i^\varepsilon(r)\,\mathbf 1_{(r_{u_{i}^\varepsilon},1)} \qquad\text{in }\mathcal D'(0,1),
\end{equation}
and
\begin{equation}\label{eq:delta-measure-v}
\frac{d}{dr}\!\left(
r^{n-1}\mathcal A_i^v\bigl(r,(v_i^\varepsilon)'(r)\bigr)
\right)
=
r^{n-1}g_i^\varepsilon(r)\,\mathbf 1_{(r_{v_{i}^\varepsilon},1)} \qquad\text{in }\mathcal D'(0,1).
\end{equation}
\end{lemma}

\begin{proof}
We present the proof for $u_i^\varepsilon$, as the proof for $v_i^\varepsilon$ follows identically. 
Let 
\[
\mu_{u_i^\varepsilon} :=
r^{n-1}f_i^\varepsilon
-
\frac{d}{dr}\!\left(
r^{n-1}\mathcal A_i^u\bigl(r,(u_i^\varepsilon)'(r)\bigr)
\right)
\quad \in D'(0,1),
\] 
be a distribution, defined by
\[
\langle \mu_{u_i^\varepsilon},\eta\rangle
:=
\int_0^1
\mathcal A_i^u\bigl(r,(u_i^\varepsilon)'(r)\bigr)\eta'(r)\,r^{n-1}\,dr
+
\int_0^1
f_i^\varepsilon(r)\eta(r)\,r^{n-1}\,dr
\]
for all $\eta\in C_0^\infty(0,1)$.

We first verify that $\mu_{u_i^\varepsilon}$ is a nonnegative Radon measure. 
Let $\eta\in C_0^\infty(0,1)$ with $\eta\ge0$. Since $u_i^\varepsilon\ge0$ a.e., the perturbed function
$\phi=u_i^\varepsilon+\eta $
remains a valid competitor in  $\mathcal K_i^{u,\varepsilon}$. Testing the variational inequality
\eqref{eq:variational-radial-ui-compact} with $\phi$ immediately gives
\[
\langle \mu_{u_i^\varepsilon},\eta\rangle\ge0
\qquad
\forall \eta\in C_0^\infty(0,1),\ \eta\ge0.
\]
Hence $\mu_{u_i^\varepsilon}$ is a positive distribution, and therefore a nonnegative Radon measure on $(0,1)$.

Since, in the interior of the support of the solution, one may perform
variations in both directions by considering
\[
\phi=u_i^\varepsilon\pm\eta,
\qquad
\eta\in C_0^\infty(\{u_i^\varepsilon>0\}),
\]
it follows that
\[
\mu_{u_i^\varepsilon}=0
\qquad\text{in}\qquad
\{u_i^\varepsilon>0\}\cup\{u_i^\varepsilon=0\}^{\circ}.
\]
Hence $\mu_{u_i^\varepsilon}$ is a nonnegative Radon measure supported on
$\partial\{u_i^\varepsilon>0\}.$

Let now $\eta\in C_0^\infty(0,1)$, $\eta\ge0$, and define
\[
\eta_\delta=\eta\,H_\delta,
\]
where
\[
H_\delta=
\begin{cases}
1, & \text{if } u_i^\varepsilon\ge 2\delta,\\[4pt]
\dfrac{u_i^\varepsilon}{\delta}-1,
& \text{if } \delta<u_i^\varepsilon<2\delta,\\[8pt]
0, & \text{if } u_i^\varepsilon\le \delta.
\end{cases}
\]

Then
\[
-\langle \eta_\delta,f_i^\varepsilon r^{n-1} \rangle
=
\int_{\{u_i^\varepsilon>0\}}
A_i^u\!\bigl(r,(u_i^\varepsilon)'\bigr)
\,\eta_\delta'(r) r^{n-1} \,dr.
\]

Since
\[
\eta_\delta'
=
H_\delta\,\eta'
+
\frac{1}{\delta}\,
\chi_{\{\delta<u_i^\varepsilon<2\delta\}}
\,(u_i^\varepsilon)'\eta,
\]
we obtain
\begin{eqnarray}
 & -\langle \eta_\delta,f_i^\varepsilon r^{n-1} \rangle
=&
\int_0^1
A_i^u\!\bigl(r,(u_i^\varepsilon)'(r)\bigr)
\,\eta'(r)\, 
H_\delta(r)\, r^{n-1} dr\nonumber\\ 
&&+
\frac1\delta
\int_{\{\delta<u_i^\varepsilon<2\delta\}}
A_i^u\!\bigl(r,(u_i^\varepsilon)'(r)\bigr)
\,(u_i^\varepsilon)'(r)
\,\eta(r)\, r^{n-1}dr.\nonumber
\end{eqnarray}

By the monotonicity assumption
$ A_i^u(r,\xi)\,\xi\ge0, $
the second term is nonnegative, and therefore
\[
-\langle \eta_\delta,f_i^\varepsilon r^{n-1} \rangle
\ge
\int_0^1
A_i^u\!\bigl(r,(u_i^\varepsilon)'\bigr)
\,\eta'\,
H_\delta\, r^{n-1} dr.
\]

Passing to the limit as $\delta\to0$, which is justified since
$0\le\eta_\delta\le\eta$ and
\[\int_0^1
\Bigl|
A_i^u\!\bigl(r,(u_i^\varepsilon)'\bigr)
\,\eta'
\Bigr|\, r^{n-1} dr<\infty,
\]
we obtain
\[
-\langle \eta,f_i^\varepsilon r^{n-1}\rangle
\ge
\int_{\{u_i^\varepsilon>0\}}
A_i^u\!\bigl(r,(u_i^\varepsilon)'\bigr)
\,\eta'\, r^{n-1}dr.
\]

By \eqref{eq:interval}, 
$(u_i^\varepsilon)'=0$  everywhere (besides possibly  at $r_{u^\varepsilon_i}$) the right-hand side may also be written as
\[
\int_0^1
A_i^u\!\bigl(r,(u_i^\varepsilon)'\bigr)
\,\eta'\, r^{n-1}dr.
\]
Therefore
\[
-\langle \eta,f_i^\varepsilon r^{n-1}\rangle \geq 
\int_0^1
A_i^u\!\bigl(r,(u_i^\varepsilon)'\bigr)
\,\eta'\, r^{n-1} dr,
\]
which implies
\[
\langle \eta, \mu_{u_i^\varepsilon}  \rangle = 
\int_0^1
A_i^u\!\bigl(r,(u_i^\varepsilon)' \bigr) 
\,\eta'\, r^{n-1}dr + 
\langle \eta,f_i^\varepsilon r^{n-1} \rangle \leq 0.
\]
Therefore, $\mu_{u_i^\varepsilon} $ is a nonpositive Radon measure. But since it was also a non-negative measure, we have $\mu_{u_i^\varepsilon} \equiv 0$. This implies that equation \eqref{eq:delta-measure-u} holds. 
A similar argument can be done for $v_i^\varepsilon$ to deduce \eqref{eq:delta-measure-v}.
\end{proof}

%%%%%%%%%%%%%%%%%%%%%%%%%%%%%%%%%%%%%%%%%%%%%%%
\subsection{Uniform  $C^{1,\alpha}$ bounds for $\varepsilon$-problem}

We first note that the monotonicity of $u_i^\varepsilon, v_i^\varepsilon$ implies they are bounded. However, we shall also present a simple variational proof of this bound in the next lemma.

\begin{lemma}\label{lem:uniform-Linfty-radial}
For every $i\ge1$ and every $\varepsilon>0$,
\[
0\le u_i^\varepsilon(r)\le M_1,
\qquad
0\le v_i^\varepsilon(r)\le M_2
\qquad\text{for a.e. }r\in[0,1].
\]
\end{lemma}

\begin{proof} 
Since $u_i^\varepsilon\in \mathcal K_i^{u,\varepsilon}$, we already have
$
u_i^\varepsilon\ge0
$ a.e. in $(0,1).$ Define
\[
\widetilde u_i^\varepsilon(r):=\min\{u_i^\varepsilon(r),M_1\}.
\]
Then $\widetilde u_i^\varepsilon\in \mathcal K_i^{u,\varepsilon}$, because
$\widetilde u_i^\varepsilon\ge0$ and
$\widetilde u_i^\varepsilon(1)=M_1$.

Moreover $
|(\widetilde u_i^\varepsilon)'(r)|\le |(u_i^\varepsilon)'(r)|
$ 
and since $f_i^\varepsilon\ge0$ one has
$
f_i^\varepsilon(r)\,\widetilde u_i^\varepsilon(r)
\le
f_i^\varepsilon(r)\,u_i^\varepsilon(r)
$. 
Therefore
$
\mathcal J_i^\varepsilon(\widetilde u_i^\varepsilon)
\le
\mathcal J_i^\varepsilon(u_i^\varepsilon).
$
Since $u_i^\varepsilon$ is a minimizer of $\mathcal J_i^\varepsilon$ over
$\mathcal K_i^{u,\varepsilon}$, it follows that
$\widetilde u_i^\varepsilon$ is also a minimizer. By uniqueness of the minimizer,
$u_i^\varepsilon=\widetilde u_i^\varepsilon$,
hence
$
u_i^\varepsilon(r)\le M_1
$ for a.e. $r\in(0,1).
$

The proof for $v_i^\varepsilon$ is identical, using the functional
$\mathcal I_i^\varepsilon$ and the truncation
$\widetilde v_i^\varepsilon:=\min\{v_i^\varepsilon,M_2\}$.
\end{proof}

\begin{remark}
Alternatively, the uniform upper bounds in Lemma \ref{lem:uniform-Linfty-radial} can be deduced via the maximum principle. Note that since the iterative source terms $f_i^\varepsilon$ and $g_i^\varepsilon$ are non-negative, the solutions act as subsolutions to the corresponding multi-phase operators. The bounds then follow directly from the comparison and maximum principles for $(p,q)$-Laplacian operators (see, e.g., \cite{FanZhao2001}).
For completeness, we give the above variational truncation argument, which is
self-contained and does not require invoking  maximum principle result.

\end{remark}

Next, we prove uniform local $C^{1,\alpha}$ bounds for the radial iterations.

\begin{proposition}\label{prop:C1a-radial}
Fix $\varepsilon>0$.
Then there exist $\alpha\in(0,1)$ and a constant $\mathbf{C}_{\varepsilon}>0$, independent of $i$, such that
\[
\|u_i^\varepsilon\|_{C^{1,\alpha}(0,1)}
+
\|v_i^\varepsilon\|_{C^{1,\alpha}(0,1)}
\le \mathbf{C}_{\varepsilon}
\qquad\text{for all } i\ge1.
\]
\end{proposition}
\begin{proof}
We argue only for $u_i^\varepsilon$, the proof for $v_i^\varepsilon$ is identical.
We first observe from \eqref{eq:delta-measure-u} that the Euler--Lagrange equation satisfied by $u_i^\varepsilon$ can be written in radial form as
\[
\frac{d}{dr}\!\left(
r^{n-1}\mathcal A_i^u\bigl(r,(u_i^\varepsilon)'(r)\bigr)
\right)
=
r^{n-1} f_i^\varepsilon(r) \mathbf 1_{(r_{u_{i}^\varepsilon},1)}
\quad\text{in }
\mathcal D'(0,1). 
\] 

 Fix
$r\in(r_{u_i^\varepsilon},1)$, and choose
$\rho<r_{u_i^\varepsilon}$. Since
$u_i^\varepsilon\equiv0$ on $(0,r_{u_i^\varepsilon})$, we have
$(u_i^\varepsilon)'(\rho)=0$
, and so $
\rho^{n-1}
\mathcal A_i^u\bigl(\rho,(u_i^\varepsilon)'(\rho)\bigr)=0.
$

Integrating the radial equation from  $\rho$ to $r$, and using
$0\le f_i^\varepsilon\le \Lambda_\varepsilon$, we obtain
\begin{eqnarray}
   r^{n-1}\mathcal A_i^u\bigl(r,(u_i^\varepsilon)'(r)\bigr)
&=&
\int_{\rho}^{r}
s^{n-1} f_i^\varepsilon(s)
\mathbf 1_{(r_{u_i^\varepsilon},1)}(s)\,ds\nonumber\\
&\le&
\Lambda_\varepsilon
\int_{r_{u_i^\varepsilon}}^r s^{n-1}\,ds \le
\Lambda_\varepsilon r^{n-1}(r-r_{u_i^\varepsilon}),\nonumber
\end{eqnarray}
where in the last inequality we have used $s\le r$ for $s\in(r_{u_i^\varepsilon},r)$. It follows that 
\begin{equation}\label{AiuBound}
 \mathcal A_i^u\bigl(r,(u_i^\varepsilon)'(r)\bigr)
\le
\Lambda_\varepsilon (r-r_{u_i^\varepsilon}).   
\end{equation}

Using the definition of $\mathcal A_i^u$ and 
$(u_i^\varepsilon)'\ge0$, we have
\[
\mathcal A_i^u\bigl(r,(u_i^\varepsilon)'(r)\bigr)
=
\bigl((u_i^\varepsilon)'(r)\bigr)^{p_i^u(r)-1}
+
\bigl((u_i^\varepsilon)'(r)\bigr)^{q_i^u(r)-1}
\ge
\bigl((u_i^\varepsilon)'(r)\bigr)^{p_i^u(r)-1}.
\]
Combining this with \eqref{AiuBound}, we find
$\bigl((u_i^\varepsilon)'(r)\bigr)^{p_i^u(r)-1}
\le
\Lambda_\varepsilon (r-r_{u_i^\varepsilon}),
$
and hence
\[|(u_i^\varepsilon)'(r)|
\le
\Lambda_\varepsilon^{1/(p_i^u(r)-1)}
(r-r_{u_i^\varepsilon})^{1/(p_i^u(r)-1)}.
\]

Since $0<r-r_{u_i^\varepsilon}<1$, we have
$(r-r_{u_i^\varepsilon})^{1/(p_i^u(r)-1)}
\le
(r-r_{u_i^\varepsilon})^{1/(p^*-1)}.
$
Therefore, 
\begin{equation}\label{uid-decay}
|(u_i^\varepsilon)'(r)|
\le
C_\varepsilon
(r-r_{u_i^\varepsilon})^{\alpha}
\qquad
\text{for }r\in(r_{u_i^\varepsilon},1),
\end{equation}
where $C_\varepsilon:=
\max\left\{1,\Lambda_\varepsilon^{1/(p_*-1)}\right\}$ and $\alpha = 1/(p^* - 1)$.

To establish the full H\"{o}lder continuity of $(u_i^\varepsilon)'$, let $r, s \in (0,1)$ with $r > s$.

\medskip
\noindent
 \begin{itemize}
\item \underline{\textit{Case 1: $s \le r_{u_i^\varepsilon} < r$.}}

In this case, $(u_i^\varepsilon)'(s) = 0$. Since $r - r_{u_i^\varepsilon} \le r - s$, estimate \eqref{uid-decay} immediately yields
\[
|(u_i^\varepsilon)'(r) - (u_i^\varepsilon)'(s)| = |(u_i^\varepsilon)'(r)| \le C_\varepsilon (r - r_{u_i^\varepsilon})^\alpha \le C_\varepsilon |r - s|^\alpha.
\]

\medskip
\noindent
\item\underline{\textit{Case 2: $r_{u_i^\varepsilon} < s < r$.}}

Let $d := s - r_{u_i^\varepsilon} > 0$ be the distance from $s$ to the free boundary. We distinguish two subcases:

\noindent
For the case where $|r - s| \ge d/2$, 
using the triangle inequality and \eqref{uid-decay}
\[
|(u_i^\varepsilon)'(r) - (u_i^\varepsilon)'(s)| \le |(u_i^\varepsilon)'(r)| + |(u_i^\varepsilon)'(s)| \le C_\varepsilon \left( (r - r_{u_i^\varepsilon})^\alpha + d^\alpha \right).
\]
Since $r - r_{u_i^\varepsilon} = d + |r - s| \le 3|r - s|$, we obtain
\[
|(u_i^\varepsilon)'(r) - (u_i^\varepsilon)'(s)| \le C_\varepsilon \left( 3^\alpha |r - s|^\alpha + 2^\alpha |r - s|^\alpha \right) \le \mathbf{C}_\varepsilon |r - s|^\alpha.
\]

\noindent
For the case where $|r - s| < d/2$, 
the function $u_i^\varepsilon$ is strictly positive on the interval $[s, r]$, and bounded away from the free boundary point $r_{u_i^\varepsilon}$. Integrating \eqref{eq:delta-measure-u} over $[s, r]$ gives
\[
|\mathcal A_i^u\bigl(r,(u_i^\varepsilon)'(r)\bigr)- \mathcal A_i^u\bigl(s,(u_i^\varepsilon)'(s)\bigr)| \le \frac{\Lambda_\varepsilon}{s^{n-1}} \int_s^r \tau^{n-1} d\tau \le C_\varepsilon |r - s|.
\]
Since $d \le \tau - r_{u_i^\varepsilon} \le 1$, the gradient $(u_i^\varepsilon)'(\tau) \sim d^\alpha$ is bounded away from zero. Applying the local $C^\alpha$ inverse continuity of the operator $\xi \mapsto \xi^{p_i^u-1} + \xi^{q_i^u-1}$ on this non-degenerate range converts the flux continuity into gradient continuity
\[
|(u_i^\varepsilon)'(r) - (u_i^\varepsilon)'(s)| \le C_\varepsilon |r - s|^\alpha.
\]
\end{itemize}

Combining all cases proves that $\frac{|(u_i^\varepsilon)'(r) - (u_i^\varepsilon)'(s)|}{|r - s|^\alpha} \le \mathbf{C}_\varepsilon$ for all $r, s \in (0,1)$, which completes the proof.
\end{proof}

The uniform $C^{1,\alpha}_{\mathrm{loc}}(0,1)$ bound established in Proposition~\ref{prop:C1a-radial} is strictly necessary to guarantee the convergence of the iterative sequence as $i \to \infty$. Indeed, because the variable exponents $p_i^u(r)$ and $q_i^u(r)$ depend nonlinearly on the state $u_{i-1}^\varepsilon$ and its gradient $|(u_{i-1}^\varepsilon)'|$, weak convergence in $W^{1,p_*}(0,1)$ alone is insufficient to identify the limiting operator. By the Arzel\`a--Ascoli Theorem, the equicontinuity of $\{(u_i^\varepsilon)', (v_i^\varepsilon)'\}$ in $C^\alpha_{\mathrm{loc}}(0,1)$ yields a subsequence (still denoted by $i$) such that
\[
(u_i^\varepsilon, (u_i^\varepsilon)') \longrightarrow (u^\varepsilon, (u^\varepsilon)'), \qquad (v_i^\varepsilon, (v_i^\varepsilon)') \longrightarrow (v^\varepsilon, (v^\varepsilon)')
\]
uniformly on compact subintervals of $(0,1)$. This strong $C^1_{\mathrm{loc}}$ convergence ensures that the exponent maps satisfy $p_i^u(r) \to p\left(r, u^\varepsilon(r), |(u^\varepsilon)'(r)|\right)$ and $q_i^u(r) \to q\left(r, u^\varepsilon(r), |(u^\varepsilon)'(r)|\right)$ pointwise, which allows us to pass to the limit inside $\mathcal{A}_i^u$ and conclude that the limit pair $(u^\varepsilon, v^\varepsilon)$ is a weak solution to the $\varepsilon$-regularized system.

\subsection{Uniform $\varepsilon$-independent $C^1$-smoothness}

We first prove that solutions to the $\varepsilon$-dependent problem are uniformly Lipschitz. 

\medskip

\noindent
In view of the mean-value theorem of Calculus, 
a direct consequence of this universal Lipschitz smoothness is that there exists $r_0\in(0,1)$,
independent of $\varepsilon$, such that
\[
u^\varepsilon>0,
\qquad
v^\varepsilon>0
\qquad \text{in } (r_0,1),
\]
or equivalently 
\[
\max\{r_{u^\varepsilon},r_{v^\varepsilon}\}\le r_0 ,
\]
where 
\[
r_{u^\varepsilon}:=\inf\{r\in(0,1):u^\varepsilon(r)>0\},
\qquad
r_{v^\varepsilon}:=\inf\{r\in(0,1):v^\varepsilon(r)>0\}.
\]

For notational convenience, throughout the proof of the lemma we suppress the dependence on the regularization parameter $\varepsilon$. Thus, we write.
\[
u:=u^\varepsilon,
\qquad
v:=v^\varepsilon,
\]
and similarly omit the superscript $\varepsilon$ from all associated quantities. All estimates obtained below are uniform with respect to $\varepsilon$.

On the positivity intervals and in the distributional sense, we have
\begin{equation}\label{flux}
   \bigl(r^{n-1}\Phi_{u}(r)\bigr)'
=
r^{n-1}h_\varepsilon(r,u,v),
\qquad
\bigl(r^{n-1}\Psi_{v}(r)\bigr)'
=
r^{n-1}k_\varepsilon(r,u,v), 
\end{equation}
where
\[
\Phi_{u}(r)
=
|u'(r)|^{p(r,u,|u'|)-2}u'(r)
+
|u'(r)|^{q(r,u,|u'|)-2}u'(r),
\]
and $\Psi_{v}$ is defined analogously.

From the structural assumptions \textup{(P2)} (see \eqref{eq:hk-structure}), whenever $u' \geq 0$  we have 
\begin{equation}\label{eq:bounds1}
C_1u' \leq (\Phi_{u})^s  \leq C_2 u',
\end{equation}
and similarly whenever $v' \geq 0$  we have 
\begin{equation}\label{eq:bounds2}
C_1v' \leq (\Psi_{v})^s  \leq C_2 v',
\end{equation}
where $s= 1/(p^* -1)$. Also using \eqref{eq:hk-structure1} we have 
\begin{equation}
\frac{\lambda_0}{\Lambda_0} \Phi_u(r) \le 
\Psi_v(r) \le \frac{\Lambda_0}{\lambda_0} \Phi_u(r),
\end{equation}
 provided $r_u=r_v$, see \eqref{eq:r=v}.

Since for fixed
$\varepsilon>0$, the non-negative functions $u,v\in C^{1,\alpha} (0,1)$  vanish at their respective free-boundary points, 
we must have 
$u'(r_{u})=0,
~
v'(r_{v})=0.
$
Consequently,
\begin{equation}\label{fluxatfree}
 \Phi_{u} (r_{u} )=0,\qquad \Psi_{v} (r_{v} )=0.   
\end{equation}
Moreover, the right-hand sides in \eqref{flux} are positive, and therefore the maps
\[
r\mapsto r^{n-1}\Phi_{u} (r),
\qquad
r\mapsto r^{n-1}\Psi_{v} (r)
\]
are nondecreasing on the corresponding positivity intervals.  Together with \eqref{fluxatfree},
this gives 
\begin{equation}
 \Phi_{u}\ge0
\quad\text{on }(r_{u},1),
\qquad
\Psi_{v}\ge0
\quad\text{on }(r_{v},1).   
\end{equation}
Since maps $\Phi_{u}$ and $\Psi_{v}$ vanish only when the derivative vanishes and preserve
the sign of the derivative, we obtain
\begin{equation}
 u'\ge 0 \quad\text{in }(r_{u} ,1),
\qquad
v'\ge0 \quad\text{in }(r_{v} ,1).   
\end{equation}

\begin{proposition}\label{prop:uniform-Lip}
Assume \textup{(P1)}--\textup{(P2)}. Then the solution pair  $(u^\varepsilon, v^\varepsilon)$ is uniformly Lipschitz. More precisely there is a constant $C $ (independent of $\varepsilon$) such that 
$$
0 \leq (u^\varepsilon)' \leq C, \qquad 0 \leq 
(v^\varepsilon)' \leq C.
$$
The constant $C$ may depend only on the structural data of the problem,
namely $p_*, p^*, \gamma_1, \gamma_2, M_1, M_2$, and the constants
appearing in \textup{(P2)}, but not on $\varepsilon$. 
\end{proposition}

\begin{proof} 
We first prove uniform Lipschitz bounds for our solution pair.  Since
\[
u(r_u)=0,\qquad u(1)=M_1, 
\]
the mean value theorem yields a point $t_{*u}\in(r_u,1)$ such that
\[
u'(t_{*u})=\frac{M_1}{1-r_u}.
\]
Likewise, there exists $t_{*v}\in(r_v,1)$ satisfying
\(
v'(t_{*v})=\frac{M_2}{1-r_v}.
\)
Now define
\[
t_u:=\inf\{r<t_{*u}: u'(r)=M_1\},\qquad
t_v:=\inf\{r<t_{*v}: v'(r)=M_2\}.
\]
The existence of $t_u$ and $t_v$ follows from the continuity of $u'$ and $v'$ and the intermediate value theorem, since
\[
u'(r_u)=v'(r_v)=0,
\]
while
\[
0<M_1<\frac{M_1}{1-r_u},
\qquad
0<M_2<\frac{M_2}{1-r_v}.
\]
So by definitions of $t_u$ and $t_v$,  one has $v'(t_u) = M_1 > 0$ and $v'(t_v) = M_2 > 0$, which imply that $\Phi_u(t_u),\Psi_v(t_v) \ge c_0 > 0$ for a positive constant $c_0$ depending only on $M_1, M_2$ and structural data. Since both functions are non-decreasing on their positivity sets, it follows that $\Phi_u(r) \ge \Phi_u(t_u) \ge c_0$ for all $r \in [t_u, 1]$ and the same inequality for $\Psi_v$. Combining these lower bound with structural inequality  \eqref{eq:bounds2} we obtain 
\begin{equation}\label{1lephi}
   1 \le \frac{1}{c_0^s} \Phi_u(r)^s \le \frac{C_2}{c_0^s} u'(r),\qquad \text{for all } r \in [t_u, 1],
\end{equation}
and 
\begin{equation}\label{1lepsi}
   1 \le \frac{1}{c_0^s} \Psi_v(r)^s \le \frac{C_2}{c_0^s} v'(r),\qquad \text{for all } r \in [t_v, 1].
\end{equation}

Assume for definiteness that $t_u\le t_v$.
The opposite ordering is treated identically.

We distinguish the possible sign configurations of the exponents
$\gamma_1$ and $\gamma_2$.

\medskip

\noindent
\underline{\bf Case 1: $\gamma_j \geq 0$.} 

As remarked earlier, the case where both $\gamma_j = 0$ reduces the  system to a pair of uncoupled problems, thereby reducing the system to the classical scalar setting.
In any case, we have that 
\[
u^{\gamma_1}v^{\gamma_2}
\le
M_1^{\gamma_1}M_2^{\gamma_2}.
\]
Hence, by \textup{(P2)}  there exists a constant $C>0$,
independent of $\varepsilon$, such that
\[
0\le
h_\varepsilon(r,u,v),
\,
k_\varepsilon(r,u,v)
\le C
\]
throughout the positivity region.
Integrating the second identity in \eqref{flux} over $[t_v,r]$, we obtain
\[
r^{n-1}\Psi_v(r)
-
t_v^{\,n-1}\Psi_v(t_v)
=
\int_{t_v}^{r}
\tau^{n-1}
k_\varepsilon(\tau,u,v)\,d\tau.
\]
Using the above bound and the fact that $\tau^{n-1}\le1$, it follows that
\[
r^{n-1}\Psi_v(r)
-
t_v^{\,n-1}\Psi_v(t_v)
\le C(r-t_v)
\le C.
\]
Since $\Psi_v(t_v)$ is bounded by construction, we conclude that
\[
r^{n-1}\Psi_v(r)\le C
\qquad\text{for }r\in[t_v,1].
\]
As $r\ge t_v$, division by $r^{n-1}$ yields
\[
\Psi_v(r)\le C
\qquad\text{for }r\in[t_v,1],
\]
which implies 
$
v'(r)\le C$, using \eqref{eq:bounds2}.
Since $v'\le M_2$ on $(r_v,t_v)$ by the definition of $t_v$, taking a larger constant $C$, if it is needed, leads us to have
\[
v'(r)\le C
\qquad\text{for all }r\in(r_v,1).
\]

The same argument applies to $u$, and the conclusion follows in this case
\[
u'(r)\le C,
\qquad
v'(r)\le C, \qquad \forall \ r \in (0,1).
\]

\medskip

\noindent
\underline{\bf Case 2: $\gamma_1 \geq  0 > \gamma_2 > -1$.}

In this case, we have 
\[
u^{\gamma_1}v^{\gamma_2}
\le
M_1^{\gamma_1}v^{\gamma_2}.
\]

Combining this bound with structural inequality \eqref{1lepsi}, we obtain 
\[
\int_{t_v}^{r}
u^{\gamma_1}v^{\gamma_2}\,d\tau\le M_1^{\gamma_1} \int_{t_v}^{r} v^{\gamma_2} \, d\tau 
\le
\frac{M_1^{\gamma_1}C_2}{c_0^s}
\int_{t_v}^{r}
v'v^{\gamma_2}\,d\tau
=
\frac{M_1^{\gamma_1}C_2}{c_0^s(\gamma_2+1)}
\Bigl(
(v(r))^{\gamma_2+1}
-
(v(t_v))^{\gamma_2+1}
\Bigr).
\]
Since $\gamma_2>-1$, the right-hand side is
uniformly bounded. 
The subsequent estimates are identical to those given in Case 1.

\medskip

\noindent
\underline{\bf Case 3: $0 > \gamma_j $ and $\gamma_1 + \gamma_2 > -1$.}

Integrating \eqref{flux} for $v$ over $[t_v,r]$, with
$r\in (t_v,1)$, we obtain
\[
r^{n-1}\Psi_{v}(r)
-
t_v^{\,n-1}\Psi_{v}(t_v)
=
\int_{t_v}^{r}
\tau^{n-1}
k_\varepsilon
\bigl(
\tau,u(\tau),v(\tau)
\bigr)
\,d\tau .
\] 
By \textup{(P2)} equation \eqref{eq:hk-structure}, it follows that
\begin{equation}\label{befHo}
 r^{n-1}\Psi_{v}(r)
-
t_v^{\,n-1}\Psi_{v}(t_v)
\le
\Lambda_0 r^{n-1}
\int_{t_v}^{r}
u^{\gamma_1}
v^{\gamma_2}
\,d\tau .   
\end{equation}
First notice that due to the assumptions $\gamma_1 + \gamma_2 > -1$ and  $\gamma_j < 0$ we have 
\[
a = \frac{2}{1-\gamma_1+\gamma_2} > 1,
\qquad
b = \frac{2}{1+\gamma_1-\gamma_2} > 1, \qquad \frac{1}{a} + \frac{1}{b}=1.
\]
Applying Hölder's inequality with the conjugate exponents $a,b$ gives
\begin{equation}\label{eq:holder1}
    \int_{t_v}^{r}
u^{\gamma_1}v^{\gamma_2}\,d\tau
\le
\left(
\int_{t_v}^{r}
u^{a\gamma_1}\,d\tau
\right)^{1/a}
\left(
\int_{t_v}^{r}
v^{b\gamma_2}\,d\tau
\right)^{1/b}.
\end{equation}
On $[t_v,1]$, by applying inequalities \eqref{1lephi} and \eqref{1lepsi}, one gets
\[
\int_{t_v}^{r}
u^{a\gamma_1}\,d\tau
\le
C
\int_{t_v}^{r}
u'u^{a\gamma_1}\,d\tau,
\qquad
\int_{t_v}^{r}
v^{b\gamma_2}\,d\tau
\le
C
\int_{t_v}^{r}
v'v^{b\gamma_2}\,d\tau,
\]for positive constant $C$ depending on  $M_1, M_2$ and structural data.
Combining these estimates with \eqref{befHo}-\eqref{eq:holder1}, we obtain
\[
r^{n-1}\Psi_{v}(r)
-
t_v^{\,n-1}\Psi_{v}(t_v)
\le
Cr^{n-1}
\left(
\int_{t_v}^{r}
u'u^{a\gamma_1}\,d\tau
\right)^{1/a}
\left(
\int_{t_v}^{r}
v'v^{b\gamma_2}\,d\tau
\right)^{1/b}.
\]
Direct computation of the integrals, which are valid in view of the conditions $a\gamma_1+1>0$ and $b\gamma_2+1>0$, leads to
\[
r^{n-1}\Psi_{v}(r)
-
t_v^{\,n-1}\Psi_{v}(t_v)
\le
Cr^{n-1}
\left(u(r)^{a\gamma_1 + 1}\right)^{1/a}
\left(v(r)^{b\gamma_2 + 1}
\right)^{1/b}  \leq C .
\]

Since $\Psi_{v}(t_v)$ is bounded by the choice of $t_v$, we
conclude that
\[
\Psi_{v}(r)\le C
\qquad\text{for all }r\in[t_v,1].
\]
Finally, since the map
$z\longmapsto
z^{p(r,v,z)-1}
+
z^{q(r,v,z)-1}
$
is increasing on $[0,\infty)$, the boundedness of
$\Psi_v$ implies that
$v'(r)\le C
$ for all $r\in[t_v,1].
$
On the other hand, by definition of $t_v$, we have $v'\le M_2$ on $(r_{v} ,t_v)$.
This gives the uniform,   $\varepsilon$-independent bound 
\begin{equation}
    v'(r)\le C
\qquad\text{for all }r\in(0 ,1).
\end{equation}

We now derive the corresponding estimate for $u'$. Since we are in the case
$t_u\le t_v$, we first consider $r\in(t_u,t_v)$. Integrating the first
identity in \eqref{flux} over $[t_u,r]$, we obtain
\[
r^{n-1}\Phi_u(r)-t_u^{\,n-1}\Phi_u(t_u)
=
\int_{t_u}^{r}
\tau^{n-1}
h_\varepsilon(\tau,u(\tau),v(\tau))\,d\tau  \le
C
\int_{t_u}^{r}
k_\varepsilon(\tau,u(\tau),v(\tau))\,d\tau ,
\] 
where we have used the comparison of the two singular terms, as in  \textup{(P2)} equation \eqref{eq:hk-structure}.

Note that for $u=u^\varepsilon = \max\{u, \varepsilon\}$ and $v=v^\varepsilon = \max\{v, \varepsilon\}$, \eqref{eq:epsilon-approx} yields $h_\varepsilon(r, u^\varepsilon, v^\varepsilon) = h(r, u^\varepsilon, v^\varepsilon)$ and $k_\varepsilon(r, u^\varepsilon, v^\varepsilon) = k(r, u^\varepsilon, v^\varepsilon)$. By \textup{(P2)}, both regularized source terms satisfy $\lambda_0 (u^\varepsilon)^{\gamma_1} (v^\varepsilon)^{\gamma_2} \le h_\varepsilon, k_\varepsilon \le \Lambda_0 (u^\varepsilon)^{\gamma_1} (v^\varepsilon)^{\gamma_2}$, which implies $h_\varepsilon(r, u^\varepsilon, v^\varepsilon) \le \frac{\Lambda_0}{\lambda_0} k_\varepsilon(r, u^\varepsilon, v^\varepsilon)$ with constant $C = \Lambda_0/\lambda_0$ independent of $\varepsilon$.

Using the second identity in \eqref{flux} and that  $t_u \leq t_v$, we  obtain
\[
r^{n-1}\Phi_u(r)-t_u^{\,n-1}\Phi_u(t_u)
\le
C\Big(
t_v^{\,n-1}\Psi_v(t_v)
-
t_u^{\,n-1}\Psi_v(t_u)
\Big)
\le C t_v^{\,n-1}.
\]
Moreover, by the choice of $t_u$, the quantity $\Phi_u(t_u)$ is bounded
by a structural constant. Hence
\[
r^{n-1}\Phi_u(r)\le Ct_v^{\,n-1}
\qquad\text{for }r\in(t_u,t_v).
\]
Therefore, after division by $r^{n-1}$ on the relevant interval, we get
$\Phi_u(r)\le C
$ for $r\in(t_u,t_v).
$
Obviously, this implies 
$u'(r)\le C
$ in this interval.

It remains to treat $r\in[t_v,1]$. Integrating the first identity in
\eqref{flux} over $[t_v,r]$, we find
\[
r^{n-1}\Phi_u(r)-t_v^{\,n-1}\Phi_u(t_v)
=
\int_{t_v}^{r}
\tau^{n-1}
h_\varepsilon(\tau,u(\tau),v(\tau))\,d\tau .
\]
We can now derive a  similar argument as we did for the $v$ component, using \textup{(P2)} equation \eqref{eq:hk-structure}, Hölder's inequality with the conjugate exponents $a,b>1$, and the bounds
$a\gamma_1+1>0$, $b\gamma_2+1>0$, to arrive at 
\[
r^{n-1}\Phi_u(r)-t_v^{\,n-1}\Phi_u(t_v)
=
\int_{t_v}^{r}
\tau^{n-1}
h_\varepsilon(\tau,u(\tau),v(\tau))\,d\tau \leq 
C r^{n-1} \int_{t_v}^{r}
u^{\gamma_1}v^{\gamma_2}\,d\tau
\le r^{n-1} C.
\]
Consequently,
\[
r^{n-1}\Phi_u(r)-t_v^{\,n-1}\Phi_u(t_v)
\le Cr^{n-1}.
\]
Since the estimate on $(t_u,t_v)$ gives $\Phi_u(t_v)\le C$, it follows that
\[
\Phi_u(r)\le C
\qquad\text{for }r\in[t_v,1].
\]
Again, by the strict monotonicity of the structural map, we conclude that
\[
u'(r)\le C
\qquad\text{for }r\in[t_v,1].
\]
On the other hand, by the definition of $t_u$, we have
$u'\le M_1$ on $(r_u,t_u]$. Combining the estimates on
$(r_u,t_u]$, $(t_u,t_v)$, and $[t_v,1]$, there exists a constant
$C_2>0$, independent of $\varepsilon$, such that
\begin{equation}
   u'(r)\le C_2
\qquad\text{for all }r\in(r_u,1). 
\end{equation}

This concludes the proof of the proposition.

\end{proof}

\begin{corollary}\label{cor:cont} 
Assuming the hypotheses of Proposition \ref{prop:uniform-Lip} hold. Then the following statements
are true
\begin{enumerate}

\item There exists a sequence $\varepsilon_j\downarrow 0$ such that
the limit pair
$$(u,v):=\lim_{j\to\infty} (u^{\varepsilon_j}, v^{\varepsilon_j})$$
exists.

 \item The  limit pair  $(u, v)$ are $C^1(0,1)$.

\item If non of $\gamma_j $ is zero, then 
the supports of $u, v$ coincide, i.e., $r_u = r_v$, where
\[
r_u := \inf\{r \in (0,1) : u(r) > 0\}, \qquad r_v := \inf\{r \in (0,1) : v(r) > 0\}.
\] 

\end{enumerate}

\end{corollary}

\begin{proof} 
First, for notational accuracy, we need to explicitly include the $\varepsilon$-dependence of all components in Proposition \ref{prop:uniform-Lip}.

The three cases in Proposition \ref{prop:uniform-Lip} show that our solution pairs are uniformly Lipschitz on $[0,1]$ with a constant independent of $\varepsilon$, implying that $\Phi_{u^\varepsilon} (r), \Psi_{v^\varepsilon} (r )$ are uniformly bounded.
Since we also have the following integral representations
\begin{equation}\label{eq:representation}
    \Phi_{u^\varepsilon} (r) = \int_{r_{u^\varepsilon}}^r \tau^{n-1} h_{\varepsilon}(\tau , u^\varepsilon(\tau), v^\varepsilon(\tau)) d \tau  ,\qquad 
\Psi_{v^\varepsilon}(r)= \int_{r_{v^\varepsilon}}^r \tau^{n-1} k_{\varepsilon}(\tau ,u^\varepsilon(\tau), v^\varepsilon(\tau)) d \tau ,
\end{equation}
both integrals must be universally bounded, independent of $\varepsilon$. In particular, for any  $r > r_{u^\varepsilon }$, we have
$$
\int_{r}^{r+\delta} \tau^{n-1} h_{\varepsilon}(\tau , u^\varepsilon(\tau), v^\varepsilon(\tau)) d \tau = o_{\delta} (1),
$$
which in turn implies, 
using that $\Phi_{u^\varepsilon} (r)$ is increasing in $r$, 
$$ 0 \leq \Phi_{u^\varepsilon} (r + \delta ) -\Phi_{u^\varepsilon} (r ) \leq o_{\delta} (1),
$$
and hence $ \Phi_{u^\varepsilon} (r)$ is uniformly continuous in $r$, and independent of $\varepsilon$.  A similar argument works for $ \Psi_{v^\varepsilon} (r)  $.
From this, one derives uniform $C^1$ smoothness for $u^{\varepsilon}, v^{\varepsilon}$.

In particular,  for a subsequence $\varepsilon_j\downarrow 0$ and nondecreasing functions $u, v$,  
$$\lim_j (u^{\varepsilon_j}, v^{\varepsilon_j})=: (u,v) ,$$
exist. Actually, the convergence is uniform in $C^1$ space, and hence $(u,v)$ are continuously differentiable over $(0,1)$, which proves claims (1)--(2) in the corollary.

This convergence,  in turn, implies 
\begin{equation}\label{eq:r=v}
    \lim_j (r_{u^{\varepsilon_j}}, r_{v^{\varepsilon_j}})=: (r_u,r_v) 
\end{equation}
exist. 

We now want to show $r_u = r_v$. Suppose this does not hold, and $r_u < r_v$. 
Then there exist numbers $a, b$ such that $r_u < a < b < r_v$, where for all $r\in[a, b]$, we have $u(r) \ge c_0 > 0$ and $v(r)=0$.

By uniform convergence, for all sufficiently large $j$
\[
u^{\varepsilon_j}(r) \ge \frac{c_0}{2} > 0 \qquad \text{and} \qquad 0 \le v^{\varepsilon_j}(r) \le \varepsilon_j \xrightarrow{j\to\infty} 0 \qquad \text{for all } r \in [a, b].
\]
We now distinguish two different cases for the power regimes on the interval $[a, b]$:

 \begin{itemize}
\item \underline{\textit{Case a: $\gamma_2 < 0$.}}
Using the integral representation \eqref{eq:representation} and the lower structural bound in \textup{(P2)}, for any $r \in (a, b)$ and all sufficiently large $j$ (where $v^{\varepsilon_j} \le \varepsilon_j$ on $[a,b]$) we have
\[
C \ge \Phi_{u^{\varepsilon_j}}(r) \ge \int_a^r \tau^{n-1} h_{\varepsilon_j}(\tau, u^{\varepsilon_j}, v^{\varepsilon_j}) \, d\tau \ge \lambda_0 \int_a^r \tau^{n-1} (u^{\varepsilon_j})^{\gamma_1} ( \varepsilon_j)^{\gamma_2} \, d\tau.
\]
Since $u^{\varepsilon_j} \ge c_0/2 > 0$ on $[a, b]$, the term $(u^{\varepsilon_j})^{\gamma_1}$ is bounded from below away from zero by a positive constant $c_1 > 0$. 
Therefore,
\[
C \ge \Phi_{u^{\varepsilon_j}}(r) \ge \lambda_0 c_1 \varepsilon_j^{\gamma_2} \int_a^r \tau^{n-1} \, d\tau = \frac{\lambda_0 c_1}{n} (r^n - a^n) \varepsilon_j^{\gamma_2}.
\]
Using $\gamma_2 < 0$, one gets $\varepsilon_j^{\gamma_2} \to +\infty$ as $j \to \infty$. Since $r > a$, the right-hand side tends to $+\infty$, which directly contradicts the uniform upper bound $C \ge \Phi_{u^{\varepsilon_j}}(r)$.

\item \underline{\textit{Case b: $\gamma_2 > 0$.}} 
Again let $r \in (r_u, r_v)$ be arbitrary. For sufficiently large $j$, we have $r_{u^{\varepsilon_j}} < r < r_{v^{\varepsilon_j}}$, which implies $v^{\varepsilon_j}(\tau) \le \varepsilon_j$ for all $\tau \in [r_{u^{\varepsilon_j}}, r]$. Thus 
\[\Phi_{u^{\varepsilon_j}}(r) \le \Lambda_0 \varepsilon_j^{\gamma_2} \int_{r_{u^{\varepsilon_j}}}^r \tau^{n-1} \big(\max\{u^{\varepsilon_j}(\tau), \varepsilon_j\}\big)^{\gamma_1} \, d\tau.\]
Using the Lipschitz bound $u^{\varepsilon_j}(\tau) \le C(\tau - r_{u^{\varepsilon_j}})$, the integrand is bounded by $C (\tau - r_{u^{\varepsilon_j}})^{\min\{\gamma_1, 0\}} \in L^1(r_u, r)$.
Therefore, by the Lebesgue Dominated Convergence Theorem, the integral remains uniformly bounded, giving us the following
\[
\Phi_{u^{\varepsilon_j}}(r) \le C \varepsilon_j^{\gamma_2} \longrightarrow 0 \qquad \text{as } j \to \infty.
\]
Thus $\lim_{j\to\infty} \Phi_{u^{\varepsilon_j}}(r) = 0$ for all $r \in (r_u, r_v)$, forcing $u \equiv 0$ on $(r_u, r_v)$, which contradicts $r_u = \inf\{r : u(r) > 0\}$.

\end{itemize}

 A similar argument applies when $r_u > r_v$, using the representation for $\Psi_{v^\varepsilon}$, where we look at the cases when $ |\gamma_1| > 0 $. 

\medskip

\noindent

When at least one of $\gamma_j = 0$, we cannot argue as above. Indeed, in this case, both equations reduce to scalar problems: $u$ gives rise to an Alt-Phillips-type scalar problem, while $v$ solves an obstacle-type problem with the right-hand side given by the determined $u^{\gamma_1}$. Depending on the boundary values, this problem may well have a solution such that $r_u < r_v$.

This implies that when at least one of $\gamma_j = 0$, the coincidence sets do not necessarily coincide. Nevertheless, if $\gamma_2 = 0$ then $r_u < r_v$ is possible but not the reverse, and a parallel argument holds when $\gamma_1=0$, i.e., $r_v \leq  r_u$. This follows from integral representations directly.

\begin{acknowledgements}

The author thanks KTH Royal Institute of Technology for its hospitality during this research. Special thanks go to Professor Henrik Shahgholian for suggesting the problem and for continuous guidance.

\end{acknowledgements}

\end{proof}


\begin{thebibliography}{99}

\bibitem{AlamriUrbano2026}
Y.~Alamri and J.~M.  Urbano,   \emph{The two-phase Alt-Phillips problem for quasilinear operators.} (2026). \textit{arXiv preprint arXiv:2604.05245}.



\bibitem{AltPhi86}
H.~W. Alt and D.~Phillips,
\emph{A free boundary problem for semilinear elliptic equations},
J. Reine Angew. Math. \textbf{368} (1986), 63--107.


\bibitem{AraTei13}
D.~Ara\'{u}jo and E.~Teixeira,
\emph{Geometric approach to nonvariational singular elliptic equations},
Arch. Ration. Mech. Anal. \textbf{209} (2013), 1019--1054.

\bibitem{BiagiValdinociVecchi2020}
S. Biagi, E. Valdinoci, E. Vecchi,
\emph{A symmetry result for cooperative elliptic systems with singularities},
{Publ. Mat.} \textbf{64} (2020), 621--652.


\bibitem{Chen2006}
Y.~Chen, S.~Levine, and M.~ Rao, 
\emph{Variable exponent, linear growth functionals in image restoration},
SIAM Journal on Applied Mathematics, \textbf{66} (2006), 1383--1406.




\bibitem{Diening2011}
L.~Diening, P.~Harjulehto, P.~Hästö, and M.~Růžička, 
\emph{Lebesgue and Sobolev Spaces with Variable Exponents}, 
Lecture Notes in Mathematics 2017, Springer, Heidelberg, 2011.



\bibitem{ElShah}
L.~El Hajj and H.~Shahgholian, 
\emph{A free boundary problem for systems (The symmetric regime)}, 
Communications on Pure and Applied Analysis, \textbf{24} (2025), 1280--1295.


\bibitem{El-Je-Sh-2}
L.~El Hajj, S. Jeon,  and H.~Shahgholian, 
\emph{Existence theory for non-variational  systems  with free boundaries.} Preprint.  https://arxiv.org/abs/2607.16767


\bibitem{FanZhao2001}
X.~Fan and D.~Zhao, 
\emph{On the spaces $L^{p(x)}(\Omega)$ and $W^{m,p(x)}(\Omega)$}, 
Journal of Mathematical Analysis and Applications \textbf{263} (2001), 424--446.


\bibitem{Juli} J. Fernández Bonder, S. Martínez and  N. Wolanski,
\emph{A free boundary problem for the p(x)-Laplacian},
Nonlinear Analysis: Theory, Methods and Applications,
\textbf{72} (2010),  1078--1103.



\bibitem{Ferra} F. ~Ferrari, M.~ Jacob and C. ~Lederman,  \emph{Two-phase free boundary problems for operators with nonstandard growth}, La Matematica 5, \textbf{27} (2026), https://doi.org/10.1007/s44007-026-00202-3


\bibitem{Harj2007} P. Harjulehto, P. Hästö, M. Koskenoja, T. Lukkari and N. Marola, \emph{An obstacle problem and superharmonic functions with nonstandard growth}, Nonlinear
Anal. \textbf{67} (2007), 3424–-3440.


\bibitem{Led} C.  Lederman and  N. Wolanski,
\emph{Inhomogeneous minimization problems for the p(x)-Laplacian},
Journal of Mathematical Analysis and Applications,
\textbf{475} (2019),  423--463.

\bibitem{MDS-Nor} E. Moreira dos Santos and G. Nornberg, \emph{Symmetry properties of positive solutions for fully nonlinear elliptic systems}, J. Differential Equations \textbf{269} (2020), 4175-–4191.



\bibitem{Ruzicka2000}
M.~Růžička, 
\emph{Electrorheological Fluids: Modeling and Mathematical Theory}, 
Lecture Notes in Mathematics, Vol.~1748, Springer, Berlin, 2000.
\end{thebibliography}
\end{document}